\documentclass{amsart}[12pt]
\usepackage{amsmath, amsthm, amscd, amsfonts,}
\usepackage{graphicx,float}

\usepackage{amssymb}

\usepackage{pb-diagram}
\usepackage{lamsarrow}
\usepackage{pb-lams}

\newcommand{\PP}{{\mathbb{P}}}
\newcommand{\QQ}{{\mathbb{Q}}}
\newcommand{\RR}{{\mathbb{R}}}

\DeclareMathOperator{\dom}{dom}

\DeclareMathOperator{\Add}{Add}

\DeclareMathOperator{\REG}{REG}
\DeclareMathOperator{\CARD}{CARD}
\DeclareMathOperator{\ZFC}{ZFC}
\DeclareMathOperator{\GCH}{GCH}
\DeclareMathOperator{\supp}{supp}
\DeclareMathOperator{\Ult}{Ult}
\DeclareMathOperator{\limdir}{lim\ dir}

\DeclareMathOperator{\Sacks}{Sacks}

\def\MPB{{\mathbb{P}}}
\def\MQB{{\mathbb{Q}}}

\def\k{\kappa}

\def\l{\lambda}

\def\a{\alpha}

\def\b{\beta}

\newtheorem{theorem}{Theorem}[section]
\newtheorem{lemma}[theorem]{Lemma}

\newtheorem{definition}[theorem]{Definition}

\newtheorem{remark}[theorem]{Remark}
\newtheorem{claim}[theorem]{Claim}
\newtheorem{question}[theorem]{Question}
\numberwithin{equation}{section}

\usepackage{pb-diagram}

\def\l{\lambda}

\def\rmark{\mbox{$\rm\bf\rule{0.06em}{1.45ex}\kern-0.05em R$}}
\def\pmark{\mbox{$\rm\bf\rule{0.06em}{1.45ex}\kern-0.05em P$}}
\def\nmark{\mbox{$\rm\bf\rule{0.06em}{1.45ex}\kern-0.05em N$}}
\def\vdash{\mbox{$\rm\| \kern-0.13em -$}}

\def\l{\lambda}

\def\rmark{\mbox{$\rm\bf\rule{0.06em}{1.45ex}\kern-0.05em R$}}
\def\pmark{\mbox{$\rm\bf\rule{0.06em}{1.45ex}\kern-0.05em P$}}
\def\nmark{\mbox{$\rm\bf\rule{0.06em}{1.45ex}\kern-0.05em N$}}
\def\vdash{\mbox{$\rm\| \kern-0.13em -$}}
\newcommand{\lusim}[1]{\smash{\underset{\raisebox{1.2pt}[0cm][0cm]{$\sim$}}
{{#1}}}}

\begin{document}

\title[An Easton like theorem in the presence of Shelah Cardinals]{An Easton like theorem  in the presence of Shelah Cardinals}

\author[M. Golshani ]{Mohammad
  Golshani }

\thanks{The author's research has been supported by a grant from IPM (No. 91030417).} \maketitle



\begin{abstract}
 We show that Shelah cardinals are preserved under  the canonical $\GCH$ forcing notion.
We also show that if $\GCH$ holds and $F:\REG\rightarrow \CARD$ is an Easton function which satisfies some weak properties, then there exists a cofinality preserving generic extension of the universe which preserves Shelah cardinals and satisfies  $\forall \kappa\in \REG,~ 2^{\kappa}=F(\kappa)$. This gives a partial answer
to a question asked by Cody \cite{cody} and independently by Honzik \cite{honzik}.
We also prove an indestructibility result for Shelah cardinals.
\end{abstract}

\section{Introduction}
In early 90's ,   Shelah cardinals were introduced by Shelah \cite{shelah-woodin} ,  to reduce the large cardinal strength of   Lebesgue measurability and the Baire property of definable sets of reals in $L(\mathbb{R})$ from supercompact cardinals to the much smaller large cardinals .  In the same paper ,  H.  Woodin introduced another kind of large cardinals ,  now called Woodin cardinals ,  which are much weaker than Shelah cardinals ,  and deduced the same results from them .  Later  analysis of the nature of these problems by Woodin and others ,  unfolded the fact that  Woodin cardinals are the correct ones for the study of such problems .   However Shelah cardinals remained of interest for several reasons including the core model project,  which has been stopped on the border of Shelah cardinals .

Recall from  \cite{shelah-woodin} that
an uncountable cardinal $\k$ is called a Shelah cardinal, if for every $f:\k \rightarrow \k,$ there exists an elementary embedding $j: V \rightarrow M$ with $crit(j)=\k$ such that $^{\k}M \subseteq M$ and $V_{j(f)(\k)} \subseteq M$.
It is also easily seen that a cardinal $\kappa$ is a Shelah cardinal if and only if, if for every $f:\k \rightarrow \k,$ there exists an elementary embedding $j: V \rightarrow M$ with $crit(j)=\k$ such that $^{\k}M \subseteq M$ and $H(j(f)(\k)) \subseteq M.$

Though the above definition is not expressible in $\ZFC,$ but  it is easily seen that we can formalize it in $\ZFC$ using the notion of extenders, which we refer to
\cite{martin-steel} for the definition and basic properties of them. In fact,  a cardinal $\kappa$
is a Shelah cardinal iff there exists a cardinal $\lambda$ such that for any $f: \kappa \to \kappa$, there exists an extender $E \in V_\lambda$
with $crit(j_E)=\kappa$ and $V_{j_E(f)(\kappa)} \subseteq \supp(E)$, where $j_E$ is the elementary embedding induced by $E$ \footnote{Recall from \cite{martin-steel}
that if $Y$ is a transitive set and $E= \langle E_s: s \in [Y]^{<\omega} \rangle$ is an extender, then $\supp(E)=Y$.}.

Shelah cardinals lie between Woodin cardinals and superstrong cardinals in the large cardinal hierarchy.
In fact, if $\kappa$ is a Shelah cardinal, then $\kappa$ is a Woodin cardinal and there are $\kappa$-many Woodin cardinals below $\kappa.$
On the other hand if $\kappa$ is a superstrong cardinal, then $\kappa$ is a Shelah cardinal and there are $\kappa$-many Shelah cardinals below $\kappa.$

In this paper we study Shelah cardinals and their relation with the continuum function. We show that Shelah cardinals are preserved under  the canonical $\GCH$ forcing notion.
We also prove an analogue of Easton's theorem in the presence of Shelah cardinals, which partially answers a question of Cody \cite{cody} and Honzik \cite{honzik}.
Also indestructibility of Shelah cardinals under some classes of Prikry type forcing notions is proved, which is similar to the Gitik-Shelah indestructibility
result for strong cardinals \cite{gitik-shelah}.

The structure of the paper is as follows.  In section 2 we discuss the notion of witnessing number of a Shelah cardinal
which plays an important role in latter sections of the paper. In section 3 we show that Shelah cardinals are preserved under the canonical
$\GCH$ forcing, so that $\GCH$ is consistent with the existence of Shelah cardinals.  In section 4 we prove an Easton like theorem in the presence of Shelah cardinals  and finally in section 5 we prove an indestructibility result for Shelah cardinals which is similar to the Gitik-Shelah indestructibility result for strong cardinals \cite{gitik-shelah}.

The next lemma is folklore and we will use it repeatedly in the paper.
We prove a proof for completeness.
\begin{lemma}\label{dominating function lemma}
If $\kappa$ is a regular cardinal in $V$ and $\mathbb{P}$ satisfies one of the following conditions, then for every function $f:\kappa\rightarrow\kappa$ in a $\PP$-generic extension of $V,$ there exists a function $g:\kappa\rightarrow\kappa$ in $V$ such that $g$ dominates $f$, i.e., $\forall\alpha\in \kappa~~~f(\alpha)<g(\alpha)$.
\begin{enumerate}
\item[(1)] $|\mathbb{P}|<\kappa$.
\item[(2)] $\mathbb{P}$ is $\kappa$-c.c.
\item[(3)] $\mathbb{P}$ is $\kappa^{+}$-distributive.
\end{enumerate}
\end{lemma}
\begin{proof}
$(3)$ is trivial, as such a forcing notion produces no new functions $f: \k \rightarrow \k,$ and $(1)$ follows from $(2),$ so it suffices to prove $(2).$ Thus suppose that $\PP$ is $\k$-c.c., and let $f\in V^\PP$ be a function from $\k$ to $\k.$ Fix a $\PP$-name $\lusim{f}$ for $f$, and for each $\a<\k$ set
\begin{center}
$D_\a=\{p\in \PP: \exists \beta<\k, p\Vdash$``$\lusim{f}(\a)=\beta$''$\}.$
\end{center}
$D_\a$ is a dense subset of $\PP.$ Let $A_\a \subseteq D_\a$ be a maximal antichain of $\PP,$ and set
\begin{center}
$g(\a)=\sup\{\b: \exists p\in A_\a,  p\Vdash$``$\lusim{f}(\a)=\beta$''$    \}+1.$
\end{center}
$g:\k \rightarrow \k$ is well-defined (as each $A_\a$ has size $<\k$), and it dominates $f$.
\end{proof}

\section{Properties of witnessing number of Shelah cardinals}
In this section we introduce the witnessing number of a Shelah cardinal, which plays an important role in subsequent sections, and discuss some of its properties.
\begin{definition}
(Suzuki \cite{suzuki}) Given a Shelah cardinal $\k,$ the witnessing number of $\k$, denoted $wt(\kappa),$ is the least cardinal $\lambda$ such that for any $f: \k \rightarrow \k,$ there exists an extender $E\in V_\l$
witnessing the Shelahness of $\k$ with respect to $f$.
\end{definition}
\begin{remark}
\begin{enumerate}
\item Let $\k$ be a Shelah cardinal, and for each $f: \k \rightarrow \k,$
 let $E_f$ be an extender of minimal rank such that its ultrapower $j_f: V \to M \simeq \Ult(V, E_f)$ witnesses the Shelahness of $\k$ with respect to $f$.
Then $wt(\k)=\sup\{j_f(f)(\k): f: \k \rightarrow \k \};$ in particular
 $2^\kappa < wt(\k)$ $($ see Lemma $2.3(2)$$)$ and it is a singular cardinal with $\k< cf(wt(\k)) \leq 2^\k.$
\item  Let $\k$ be a Shelah cardinal, and let $\l < wt(\k).$ Then there is $f: \k \rightarrow \k$ and an elementary embedding $j: V \rightarrow M \supseteq V_{j(f)(\k)}$ with $crit(j)=\k$ such that $\l < j(f)(\k).$
\item Let $\k$ be a Shelah cardinal, $f,g: \k \rightarrow \k$ and $\forall \a<\k, f(\a) < g(\a)+\omega.$ Further suppose that $f(\a) \geq \a+1,$ for all
  $\a < \k.$ Let $j_f: V \rightarrow M_f$ and $j_g: V \rightarrow M_g$ witness the Shelahness of $\k$ with respect to $f$ and $g$ respectively. Then we can assume that  $j_g(f)(\k)=j_f(f)(\k);$ in particular $j_g$ also witnesses the  Shelahness of $\k$ with respect to $f$. To see this, let $F$ be the extender derived from
  $j_{g}$ with support $\supp(F)= V_{j_{g}(f)(\k)}$. As $j_g(f)(\k) \geq \k+1$, we have $\k \in V_{j_{g}(f)(\k)}$ and hence such an extender exists. Further we have
  $j_g(f)(\k) = j_F(f)(\k)$ and if $j_F: V \to M_F \simeq \Ult(V, F)$ is the ultrapower embedding, then  $V_{j_F(f)(\k)} \subseteq M_F.$ So by replacing $E_f$ with $F$ if necessary, we can assume that $j_g(f)(\k)=j_f(f)(\k).$
\end{enumerate}
\end{remark}

\begin{lemma}
(Suzuki \cite{suzuki}) Let $\k$ be a Shelah cardinal. Then

$(1)$ $\forall \xi < wt(\k), \k$ is a $\xi$-strong cardinal.

$(2)$ $\{\l< wt(\k): \l$ is a measurable Woodin cardinal$ \}$ is unbounded in $wt(\k).$
\end{lemma}
Thus the existence of Shelah cardinals imply the existence of large cardinals above them. In the next lemmas we show that the above results are in some sense the best we can prove. We show that it is consistent that $\k$ is not $wt(\k)$-strong and that there are no large cardinals above $wt(\k)$.
\begin{lemma}\label{Distributive forcing and Shelah cardinals}
If $\kappa$ is a Shelah cardinal, then any $wt(\kappa)$-distributive forcing notion preserves the Shelahness of $\kappa$.
\end{lemma}
\begin{proof}
Let $\PP$ be $wt(\kappa)$-distributive and let $G$ be $\PP$-generic over $V$. Let $f\in V[G], f:\k \rightarrow \k.$ Then $f\in V.$ Let $E\in V_{wt(\k)}$ be an extender witnessing the Shelahness of $\k$ with respect to $f$. Let $H\in V[G]$ be the filter on $j(\PP)$ generated by $j[G].$ By \cite{cummings}, $H$ is in fact $j(\PP)$-generic over $M$, and $j$ lifts to some $j^*: V[G] \rightarrow M[H].$ Clearly $(V[G])_{j^*(f)(\k)}=V_{j(f)(\k)} \subseteq M \subseteq M[H],$ and hence $j^*$ witnesses that $\k$ is Shelah in $V[G]$ with respect to $f$.
\end{proof}
We now give some applications of the above lemma.
\begin{lemma}
Suppose $\kappa$ is a Shelah cardinal. There exists a generic extension $V[G]$ of $V$, in which $\kappa$ remains Shelah and there are no Mahlo cardinals above $wt(\k).$
\end{lemma}
\begin{proof}
Let
\[
\MPB= \langle \langle    \MPB_\l: \l > wt(\k)      \rangle, \langle  \lusim{\MQB}_\l: \l > wt(\k)     \rangle\rangle
\]
be the reverse Easton iteration of forcing notions such that for each  ordinal $\l > wt(\k)$,
 \begin{itemize}
 \item $\Vdash_l$``$\lusim{\MQB}_\l$  is the forcing notion for shooting a club of singular cardinals through $\l,$ by its approximations of size $<\l$'', if $\l$
 is a Mahlo cardinal,
 \item $\Vdash_l$``$\lusim{\MQB}_\l$ is the trivial forcing'', otherwise.
 \end{itemize}
 For each ordinal   $\l > wt(\k), \Vdash_\l$``$\lusim{\MQB}_\l$ is $\l$-strategically closed'' and if $\l$ is Mahlo, then forcing with $\MPB_{\l+1}=\MPB_\l* \lusim{\MQB}_\l$
 makes $\l$ an inaccessible non-Mahlo cardinal.  As the full forcing $\PP$ is $wt(\k)$-strategically closed and hence $\l$-distributive, by Lemma 2.4, $\k$ remains a Shelah cardinal in $V^\PP.$ It is also clear that in the extension by $\PP,$ there are no Mahlo cardinals above $wt(\k).$
\end{proof}
In fact we can kill all inaccessible cardinals above $wt(\k)$ preserving the Shelahness of $\k.$ To do this, it suffices to add a club $C$ of singular cardinals above $wt(\k),$ and then make $2^{\alpha^{++}}=\alpha_*^+,$ where $\alpha < \alpha_*$ are two successive points in C.
To be more precise, let $\MPB$ be the following class forcing notion for adding a club of singular cardinals above $wt(\k).$
Conditions in $\MPB$ are club sets $c$ such that $\min(c)=wt(\kappa)^{+\omega}$ ordered by end extension. The forcing notion $\MPB$
is $\l$-distributive for all cardinals $\lambda$ and hence forcing with it does not add any new sets. Further, if $G$ is $\MPB$-generic over $V$
and $C= \bigcup_{c \in G}c,$ then $C$ is a club of singular cardinals above $wt(\kappa)$ with $\min(C)=wt(\kappa)^{+\omega}$. Force over
$V[G]$ with the forcing notion $\MQB$ which is the Easton support product of forcing notions $\Add(\alpha^{++}, \alpha_*^+),$ where $\alpha < \alpha_*$
are successive points in $C$. Let $H$ be $\MQB$-generic over $V[G]$. Then there are no inaccessible cardinals in  $V[G][H]$ above $wt(\k)$. To see this suppose $\l$
is an inaccessible  cardinal above $wt(\kappa).$ Then for some $\a \in C,$ we have $\a^{++} < \l < \a_*$. But then
$2^{\a^{++}} > \l$, which contradicts the strong inaccessibility of $\l.$

\begin{lemma}\label{Closed forcing and Shelah cardinals}
If  $\kappa$ is a Shelah cardinal and $\lambda<wt(\kappa)$ is a regular cardinal, then there is a $\lambda$-closed forcing $\mathbb{P}$ such that $\kappa$ is not Shelah in $V^{\mathbb{P}}$.
\end{lemma}
\begin{proof}
By Lemma 2.5 we can assume that there are no Mahlo cardinals above $wt(\k).$
Then it is easily seen, using Lemma 2.3, that $\PP=Add(\lambda, wt(\k)^+)$ forces that $\k$ is not a Shelah cardinal.
\end{proof}
\begin{lemma}
Suppose $\kappa$ is a Shelah cardinal and   $wt(\kappa)$-strong. Let $j:V\rightarrow M$ witness $wt(\kappa)$-strength of $\kappa$ with $^\kappa$$M\subseteq M$. Then $M\models$``$ \kappa$ is a Shelah cardinal'', and  there is a normal measure $U$ on $\kappa$ such that $\{\lambda<\kappa: \lambda$ is a Shelah cardinal$\}\in U$.
\end{lemma}
\begin{proof}
The first part is trivial as $V_{wt(\k)} \subseteq M.$ Now let $U=\{X \subseteq\k: \k\in j(X)\}.$ Then $U$ is a normal measure on $\k$, and $\{\lambda<\kappa: \lambda$ is a Shelah cardinal$\}\in U$.
\end{proof}
It follows that if $\k$ is a Shelah cardinal, and Shelah cardinals below $\k$ are bounded in $\k,$ then $\k$ is not $wt(\k)$-strong.
\begin{theorem}
Let $\k$ be a Shelah cardinal. The following are equivalent:

$(1)$ $\{wt(\l)<\k: \l <\k$ is a Shelah cardinal$\}$ is unbounded in $\k.$

$(2)$  $\{wt(\l)<wt(\k): \l <wt(\k)$ is a Shelah cardinal$\}$ is unbounded in $wt(\k).$

$(3)$ There exists a Shelah cardinal $\l$ with $\k < \l < wt(\l) < wt(\k).$
\end{theorem}
\begin{proof}
$(1) \Rightarrow (2):$ Let $\b < wt(\k).$ Let $f: \k \rightarrow \k$ and $j_f: V \rightarrow M_f \supseteq V_{j(f)(\k)}, crit(j_f)=\k$ be such that $\b<j_f(f)(\k).$ Now define $g: \k \rightarrow \k$ by
\begin{center}
$g(\a)=wt($the least Shelah cardinal above $\max\{\a, f(\a)   \}),$
\end{center}
and let $h(\a)=g(\a)^{+\omega+4}$ \footnote{We defined $h$ from $g$ this way to be able to apply Remark $2.2(3)$. We could define it other ways, as long as Remark $2.2(3)$ applies. The same remark applies at later similar arguments in the paper.}. Let $j_h: V \rightarrow M_h$ witness the Shelahness of $\k$ with respect to $h.$ Then by Remark $2.2(3), j_f(f)(\k)=j_h(f)(\k)$
and we have
\begin{center}
$\b < j_f(f)(\k)=j_h(f)(\k)  < j_h(g)(\k) < j_h(g)(\k)^{+\omega+4} = j_h(h)(\k),$
 \end{center}
 and
 \begin{center}
 $j_h(g)(\k)=wt($the least Shelah cardinal above $\max\{\k, j_h(f)(\k)   \}).$
 \end{center}
So $M_h\models$``there exists a Shelah cardinal, say  $\l$, in the interval $(\b, j_h(g)(\k))$'', and since $M_h \supseteq V_{j_h(h)(\k)}, \l$ is in fact a Shelah cardinal in $V$ and $wt(\l)< wt(\k).$

$(2) \Rightarrow (3):$ is trivial.

$(3) \Rightarrow (1):$ Let $\b<\k.$ Let $f: \k \rightarrow \k$ be such that $f(\a) > \max\{\b, \a\},$ and let $j: V \rightarrow M$ witness the Shelahness of $\k$ with respect to $f$. We may further suppose that $f$ is such that $j(f)(\k) > wt(\l)^{+\omega+4}.$ Now
\begin{center}
$M\models$``$\exists \l (\b < \l < wt(\l) < j(\k)$ and $\l$ is a Shelah cardinal''.
\end{center}
So by elementarily
\begin{center}
$V\models$``$\exists \l (\b < \l < wt(\l) < \k$ and $\l$ is a Shelah cardinal''.
\end{center}
The result follows.
\end{proof}

\section{GCH in the presence of Shelah cardinals}
In this section we show that the existence of Shelah cardinals is consistent with $\GCH$ (if there are any), and in fact we will show that the canonical forcing for $\GCH$ preserves Shelah cardinals.
\begin{definition}
The canonical forcing for $\GCH$ is defined as the reverse Easton iteration of forcings
\begin{center}
$\langle \langle \PP_\gamma: \gamma\in On   \rangle, \langle \lusim{\mathbb{Q}}_\gamma: \gamma\in On   \rangle \rangle$
\end{center}
where at each step $\gamma,$ if $\gamma$ is a cardinal in $V^{\PP_{\gamma}},$ then $V^{\PP_{\gamma}}\models$``$\mathbb{Q}_\gamma=\Add(\gamma^+, 1)$'', and $V^{\PP_{\gamma}}\models$``$\mathbb{Q}_\gamma$ is the trivial forcing notion'' otherwise.
\end{definition}
\begin{theorem}\label{consistency of Shelah and GCH}
The canonical forcing for $\GCH$ preserves all Shelah cardinals.
\end{theorem}
\begin{proof}
Let $G$ be $\PP$-generic over $V$, and let $f\in V[G], f: \k \rightarrow \k.$ As we can write $\PP=\PP_\k*\lusim{\PP}_{[\k, \infty)},$ where $\PP_\k$ is $\k$-c.c. and $\Vdash_{\PP_\k}$``$\lusim{\PP}_{[\k, \infty)}$ is $\k^+$-closed'', it follows from Lemma 1.1 that there is $g\in V, g: \k \rightarrow \k$ which dominates $f$. Define $h: \k \rightarrow \k$ by
\begin{center}
$h(\a)=$the least inaccessible cardinal above $\max\{\a, g(\a)  \}.$
\end{center}
Since $\k$ is a Shelah cardinal in $V$ and $h\in V,$ there exists an elementary embedding $j: V \rightarrow M$ with $crit(j)=\k$ such that $^{\k}M \subseteq M$ and $V_{j(h)(\k)}\subseteq M.$ Let $\delta=j(h)(\k).$ If $\gamma<\delta$ is a limit ordinal, then $\PP^M_\gamma=\PP_\gamma$ (where $\PP^M=j(\PP)$), however $\PP_\delta$ need not be the same as $\PP^M_\delta.$ This is because though $\delta$ is inaccessible in $M$, $\delta$ need not be inaccessible in $V$.
\begin{claim}
We can take $j$ so that $\delta$ is inaccessible in $V$.
\end{claim}
\begin{proof}
Let $\bar{h}(\a)=h(\a)+\omega+2,$ and let $i: V \rightarrow N$ witness the Shelahness of $\k$ with respect to $\bar{h}.$ Let $E$ be the $(\k, |V_{i(h)(\k)}|^+)$-extender derived from $i$ and let $j_E: V \rightarrow M_E \cong Ult(V, E)$ be the ultrapower embedding so that $crit(j_E)=\kappa$, $^\kappa$$M_E \subseteq M_E$
 and $V_{i(h)(\kappa)} \subseteq M_E.$ Then  $i(h)(\kappa)$ is inaccessible. So it suffices to take $j=j_E.$
\end{proof}
Thus let's assume from the beginning that our $j$ has this property. It follows that $\PP^M_\delta=\PP_\delta.$ Let
\begin{center}
$\PP^M_{j(\k)}=\PP_\k*\lusim{\Add}(\k^+, 1)* \lusim{\PP}_{(\k, \delta)}*\lusim{\PP}^M_{[\delta, j(\k))},$
\end{center}
where $\PP_{(\k, \delta)}$ is the forcing between $\k$ and $\delta.$ Also let $G=G_{<\k}*G(\k)*G_{(\k,\delta)}*G_{\text{tail}}$ correspond to
\begin{center}
$\PP=\PP_\k*\lusim{\Add}(\k^+, 1)* \lusim{\PP}_{(\k, \delta)}*\lusim{\PP}_{\text{tail}}.$
\end{center}
Let $i:V \rightarrow N$ be given from the transitive collapse of $\{j(g)(\k, \delta): g\in V \}$ and let $k: N \rightarrow M$ be the factor map, so that we have the following commutative diagram:
\begin{center}

\begin{align*}
\begin{diagram}
\node{V}
        \arrow{e,t}{j}
        \arrow{s,l}{i}
        \node{M}
\\
\node{N}
         \arrow{ne,b}{k}
\end{diagram}
\end{align*}

\end{center}
Note that $\Add(\k^+, 1)$ is the same in all three models and $crit(k)>\k^+,$ so we can extend $k$ to
\begin{center}
$k^*: N[G_{<\k}*G(\k)] \rightarrow M[G_{<\k}*G(\k)].$
\end{center}
Let $\delta_0=i(h)(\k),$ and note that $k(\delta_0)=\delta.$ Also let
\begin{center}
$\PP^N_{i(\k)}=\PP_\k*\lusim{\Add}(\k^+, 1)* \lusim{\PP}^N_{(\k, \delta_0)}*\lusim{\PP}^N_{[\delta_0, i(\k))},$
\end{center}
where $\PP^N=i(\PP).$ Let:
\begin{enumerate}
\item $\RR\in M[G_{<\k}*G(\k)]$ be the term forcing associated with $\PP^M_{[\delta, j(\k))}$ with respect to $\PP_{(\k, \delta)}; \RR=\PP^M_{[\delta, j(\k))}/\PP_{(\k, \delta)}.$
\item $\RR_0\in N[G_{<\k}*G(\k)]$ be the term forcing associated with $\PP^N_{[\delta_0, i(\k))}$ with respect to $\PP^N_{(\k, \delta_0)}; \RR_0=\PP^N_{[\delta_0, i(\k))}/\PP^N_{(\k, \delta_0)}.$
\end{enumerate}
Then $k^*(\RR_0)=\RR.$ Also note that
\begin{enumerate}
\item $N[G_{<\k}*G(\k)]\models$``$2^\k=\k^+$'',
\item $N[G_{<\k}*G(\k)]\models$``$\RR_0$ is $i(\k)$-c.c. of size $i(\k)$'',
\item $V[G_{<\k}*G(\k)]\models$``$\RR_0$ is $\k^+$-closed and  $|i(\k)|=\k^+$'',
\item $V[G_{<\k}*G(\k)]\models$``$^{\k}N[G_{<\k}*G(\k)] \subseteq N[G_{<\k}*G(\k)].$
\end{enumerate}
So there exists $H_0\in V[G_{<\k}*G(\k)]$ such that $H_0$ is $\RR_0$-generic over $N[G_{<\k}*G(\k)].$ Using $k^*,$ we can find $H_1$ which is $\RR$-generic over $M[G_{<\k}*G(\k)].$ Let
\begin{center}
$G^M_{[\delta, j(\k))}=\{\lusim{x}[G_{<\k}*G(\k)*G_{(\k, \delta)}]: \lusim{x}\in H_1    \}.$
\end{center}
$G^M_{[\delta, j(\k))}\in V[G_{<\delta}]$ and it is $\PP^M_{[\delta, j(\k))}$-generic over $M[G_{<\delta}].$ It follows that we can lift $j$ to
\begin{center}
$j^*: V[G_{<\k}] \rightarrow M[G_{<\delta}*G^M_{[\delta, j(\k))}].$
\end{center}
Let $E$ be the $(\k, j(h)(\k))$-extender derived from $j^*$. Then $E\in V[G_{<\delta}]$ and $E$ is in fact an extender over $V[G_{<\delta}];$ this is because forcing with $\Add(\k^+, 1)*\lusim{\PP}_{(\k, \delta)}$ does not add any new subsets of $\k$ over $V[G_{<\k}].$ So let
\begin{center}
$i^*: V[G_{<\delta}] \rightarrow N^* \cong Ult(V[G_{<\delta}], E)$
\end{center}
be the ultrapower embedding. We have
\begin{enumerate}
\item $(V_{i^*(\k)})^{ N^*} = (V_{i^*(\k)})^{Ult(V[G_{<\delta}], E)},$
\item $ (V_{j(h)(\k)})^{Ult(V[G_{<\delta}], E)}= (V_{j(h)(\k)})^{M[G_{<\delta}*G^M_{[\delta, j(\k))}]},$
\item $i^*(h)(\k)=j(h)(\k).$
\end{enumerate}
so we can conclude that $N^* \supseteq (V_{i^*(h)(\k)})^{V[G_{<\delta}]}.$ Let $i^*(\PP_{\text{tail}})=\PP^{N^*}_{\text{tail}}$ and let $G^{N^*}_{\text{tail}}$ be the filter generated by $i^*[G_{\text{tail}}].$ We show that $G^{N^*}_{\text{tail}}$ is $\PP^{N^*}_{\text{tail}}$-generic over $N^*$. So let $D\in N^*$ be dense open in $\PP^{N^*}_{\text{tail}}.$ Then $D=i^*(g)(s)$ for some $s\in [j(h)(\k)]^{<\omega}$ and $g\in V[G_{<\delta}], g: [\k]^{|s|} \rightarrow V[G_{<\delta}]$. We may suppose that $g(t)$ is dense open in $\PP_{\text{tail}},$ for all $t\in [\k]^{|s|}.$ But
\begin{center}
$V[G_{<\delta}]\models$``$\PP_{\text{tail}}$ is $\delta^+$-closed'',
\end{center}
so
\begin{center}
$V[G_{<\delta}]\models$``$\bigcap_{t}g(t)$ is dense open in $\PP_{\text{tail}}$''.
\end{center}
It follows that $G_{\text{tail}} \cap \bigcap_{t}g(t) \neq \emptyset,$ which implies $G^{N^*}_{\text{tail}} \cap D \neq \emptyset.$ So we can lift $i^*$ to get
\begin{center}
$i^{**}: V[G_{<\delta}][G_{\text{tail}}] \rightarrow N^*[G^{N^*}_{\text{tail}}].$
\end{center}
This means we have
\begin{center}
$i^{**}: V[G] \rightarrow N^*[G^{N^*}_{\text{tail}}],$
\end{center}
which is also definable in $V[G].$ Further
\begin{center}
$(V_{j(h)(\k)})^{V[G]}=(V_{i^*(h)(\k)})^{V[G]} \subseteq N^*[G^{N^*}_{\text{tail}}].$
\end{center}
Hence $i^{**}$ witnesses $\k$ is a Shelah cardinal in $V[G]$ with respect to $f$. as $f$ was arbitrary, we can conclude that $\k$ is a Shelah cardinal in $V[G],$ and the theorem follows.
\end{proof}

\section{Easton's function in the presence of Shelah cardinals}
In \cite{menas}, Menas showed using a master condition argument that locally definable
(see Definition 6.1 below) Easton functions $F$ can be realized, while
preserving supercompact cardinals. Firedman and Honzik \cite{friedman-honzik} proved the same result for strong cardinals, and  Cody \cite{cody} proved an analogous result for Woodin cardinals. Cody and independently Honzik  \cite{honzik}, asked if it is possible to prove the same result for Shelah cardinals. In this section we provide a (partial) solution to their question.

Recall that an Easton function is a definable class function $F: \REG \rightarrow \CARD$ satisfying:
\begin{enumerate}
\item $\k < \l \Rightarrow F(\k)\leq F(\l),$
\item $cf(F(\k))>\k.$
\end{enumerate}
\begin{definition}
An Easton function $F$ is said to be locally
definable if the following condition holds:

There is a sentence $\psi$  and a formula $\phi(x, y)$ with two free variables such that
$\psi$  is true in $V$ and for all cardinals $\gamma,$ if $H(\gamma)\models$``$\psi$'', then $F[\gamma] \subseteq \gamma$ and
\begin{center}
$\forall \a, \b \in \gamma$ $(F(\a)=\b \Leftrightarrow H(\gamma)\models$``$\phi(\a, \b)$''$  ).$
\end{center}
\end{definition}
\begin{theorem}
$(\GCH)$ Assume $F$ is a locally definable Easton function and let $\k$ be a Shelah cardinal such that $H(\k)\models$``$\psi$''. Then there is a cofinality preserving forcing notion $\PP$ such that $V^{\PP}$ realizes $F$ and $\k$ remains a Shelah cardinal in $V^{\PP}$.
\end{theorem}
\begin{remark}
\begin{enumerate}
\item In \cite{friedman-honzik}, the full strength of $\k$ is used to show that $\k$ is closed under $F$. As a Shelah cardinal $\k$ is not necessarily even $wt(\k)$-strong, their argument can not be applied to show the closure of  $\k$ under $F$, and since this assumption is essential in preserving  large cardinals, we added the extra assumption
$H(\k)\models$``$\psi$'' to guarantee the closure of $\k$ under $F$.

\item We can replace the assumption $H(\k)\models$``$\psi$'' with the apparently weaker assumption $\exists \k \leq \l < wt(\k), H(\l)\models$``$\psi$''. But  using the methods of section 2 (in particular the proof of Theorem 2.8), we can in fact show that the following are equivalent:
\begin{enumerate}
\item $H(\k)\models$``$\psi$''(i.e., $\k \in \mathcal{C}_\psi,$ where $\mathcal{C}_\psi$ is defined in the proof below),
\item $\exists \k \leq \l < wt(\k), H(\l)\models$``$\psi$'' (i.e. $\mathcal{C}_\psi\cap [\k,wt(\k))\neq \emptyset$),
\item $\{\l<\k: H(\l)\models$``$\psi$''$ \}$ is unbounded in $\k$ (i.e., $\mathcal{C}_\psi\cap \k$ is unbounded in $\k$).
\end{enumerate}
\end{enumerate}
\end{remark}
\begin{proof}
The forcing notion $\PP$ is essentially the forcing $\PP^F$ defined in \cite{friedman-honzik}. We refer to \cite{friedman-honzik} for  the definition of $\PP^F$ and its basic properties, and we will apply the definitions and results from it without any mention. Let $G$ be $\PP^F$-generic over $V$. By \cite{friedman-honzik}, the function $F$ is realized in $V[G]$; thus it remains to show that $\k$ remains a Shelah cardinal in $V[G].$
Let $f\in V[G], f: \k \rightarrow \k.$ First, we show that there exists $h\in V, h: \k \rightarrow \k$ which dominates $f$. We need the following (see \cite{friedman-honzik} for the definition of the forcing notion $\Sacks(\k, \l)$).
\begin{claim}
Let $f\in V^{\Sacks(\k, \l)}, f: \k \rightarrow \k.$ Then there exists $g\in V, g: \k \rightarrow \k$ which dominates $f$.
\end{claim}
\begin{proof}
Let $\vec{p}\in \Sacks(\k, \l)$ be such that $\vec{p}\Vdash$``$\lusim{f}: \k \rightarrow \k$'', where $\lusim{f}$ is a name for $f$. Build a fusion sequence $\langle   \vec{p}_\a: \a<\k \rangle$ of conditions in $\Sacks(\k, \l)$ extending $\vec{p}$, a sequence $\langle A_\a: \a<\k \rangle$ of subsets of $\k$ and a sequence $\langle X_\a: \a<\k \rangle$ of subsets of $\l$ such that:
\begin{enumerate}
\item $|A_\a|\leq (2^\a)^\gamma,$ for some $\gamma<\k,$
\item $\vec{p}_\a\Vdash$``$\lusim{f}(\a)\in A_\a$'',
\item $\bigcup_{\a<\k} X_\a=\bigcup_{\a<\k} \supp(\vec{p}_\a),$
\item $\forall \a<\k, \vec{p}_{\a+1} \leq_{\a, X_\a} \vec{p}_\a.$
\end{enumerate}
By the generalized fusion lemma, \cite{friedman-honzik} Fact 2.18, there exists $\vec{q}\in \Sacks(\k, \l),$ extending all $\vec{p}_\a$'s, $\a<\k.$ Define $g: \k \rightarrow \k$ by $g(\a)=\sup A_\a+1.$ Then $g\in V$ and $\vec{q}\Vdash$``$\forall \a<\k, \lusim{f}(\a) < g(\a)$''.
\end{proof}
Then we have
$\PP^F \cong \PP^F_\k*\lusim{\PP^F(\k)}*\lusim{\PP}^F_{\text{tail}},$
where
\begin{center}
$V^{\PP^F_\k}\models$``$\PP^F(\k)=\Sacks(\k, F(\k)) \times \prod_{ \k < \l < i_{\k+1}, \l \text{regular}} \Add(\l, F(\l))$'',
\end{center}
and $i_{\k+1}$ is the $(\k+1)$-th closure point of $F$ (note that $i_\k=\k$).
Then we have
\begin{enumerate}
\item $V^{\PP^F_\k*\lusim{\Sacks}(\k, F(\k))}\models$``$\prod_{ \k < \l < i_{\k+1}, \l \text{regular}} \Add(\l, F(\l))$ is $\k^+$-distributive'',
\item $V^{\PP^F_{i_{\k+1}}}\models$``$\PP^F_{\text{tail}}$ is $\k^+$-distributive''.
\end{enumerate}
So by Lemma 1.1 and Claim 4.4, we can find  $h\in V, h: \k \rightarrow \k$ such that for all $\a<\k, f(\a) < h(\a).$ Since $\psi$ holds in $V$, there exists a club of cardinals $\mathcal{C}_\psi$ such that if $\l\in \mathcal{C}_\psi,$ then $H(\l)\models$``$\psi$''. To see this, let $n$ be a big enough natural number which exceeds the complexity of $\psi$
in the Levy hierarchy of formulas.  By Levy's reflection principle and the fact that the $\Sigma_n$-satisfaction relation is expressible,
\begin{center}
$\mathcal{C}_\psi=\{\l: \l $ is a cardinal, $H(\l) \prec_n V$ and  $H(\l)\models$``$\psi$''  $ \}$
\end{center}
is a proper class. But clearly $\mathcal{C}_\psi$ contains its limit points, so it is in fact a club of cardinals, as required.

 Also by our assumption, $\k\in \mathcal{C}_\psi$, and hence  $\mathcal{C}_\psi\cap \k$ is a club of $\k.$
We show  that  $\mathcal{C}_\psi\cap wt(\k)$ is also  a club of $wt(\k).$ It suffices to show that
$\mathcal{C}_\psi\cap wt(\k)$ is unbounded in $wt(\k)$. Thus assume $\nu < wt(\k)$. Let $f: \k \to \k$  be an increasing continuous function such that $j_f(f)(\k) > \nu.$
Note that the set
\begin{center}
$X= \{\a < \k:$ there are unboundedly many $\beta < f(\alpha)$ with $H(\beta) \models$``$\psi$''  $\}$
\end{center}
 has measure one, so $\k \in j_f(X)$, which implies
\begin{center}
$M_f \models$`` there are unboundedly many $\beta < j_f(f)(\kappa)$ with $H^{M_f}(\beta) \models$``$\psi$''.
\end{center}
 As $H^{M_f}(\beta) = H(\beta),$ for $\beta < j_f(f)(\kappa)$, and $j_f(f)(\k) > \nu,$ we can find $\l$
 with $\nu < \l < j_f(f)(\k)$ and $H(\l)\models$``$\psi$''. Thus $\nu < \l \in \mathcal{C}_\psi\cap wt(\k)$,
 as required.

Let $\mathcal{C}_h$ be the class of closure points of $h$, and note that $\mathcal{C}_h$ is a club of cardinals, and $\mathcal{C}_h\cap \k$ is a club of $\k$. Define $h^*: \k \rightarrow \k$ by
\begin{center}
$h^*(\a)=$the $\omega$-th point of $\mathcal{C}_\psi\cap \mathcal{C}_h$ above $\max\{\a, F(\a), h(\a)  \}.$
\end{center}
Note that for each $\a<\k, h^*(\a)$ is a singular cardinal. Set $h^{\dag}(\a)=h^*(\a)^{++}.$ Since $\k$ is a Shelah cardinal in $V$, and $h^\dag\in V,$ there exists an elementary embedding $j: V \rightarrow M \supseteq H(j(h^\dag)(\k))$ with $crit(j)=\k$ and $^{\k}M \subseteq M.$ It is easily seen that:
\begin{enumerate}
\item $j(h^\dag)(\k)=(j(h^*)(\k))^{++}$,
\item $j(h^*)(\k)=$the $\omega$-th point of $j(\mathcal{C}_\psi \cap \mathcal{C}_h)$ above $\max\{\k, j(F)(\k), j(h)(\k)  \},$
\item $j(\mathcal{C}_\psi)\cap j(h^\dag)(\k)=\mathcal{C}_\psi \cap j(h^\dag)(\k),$ and so $j(h^*)(\k)\in \mathcal{C}_\psi.$
\end{enumerate}
Let $\b=j(h^*)(\k).$ Then $\b>\k$ is a singular cardinal, $H(\b)\models$``$\psi$'' and $j: V \rightarrow M \supseteq H(\b^{++})$.
We show that we can  lift $j$ to some
\begin{center}
$j^*: V[G] \rightarrow M[H] \supseteq H^{V[G]}(\b^+),$
\end{center}
for some $H$ which is $j(\PP^F)$-generic over $M$. The proof is similar to the proof of Theorem 3.17 from \cite{friedman-honzik}, and we present it here
for completeness. For notational simplicity let $\MPB=\MPB^F$ and $\MPB^M=j(\MPB)$.

  We can assume that $\beta^{++} < j(\kappa) < \beta^{+3}$ and that $M= \{ j(l)(a): l: [\kappa^{<\omega} \to V, a \in [\beta^{++}]^{<\omega}   \}$.
Since $\kappa$ is closed under $F$, $j(\kappa)$ is closed under $j(F)$. Moreover, since $j(F)$ is locally definable in $M$ via the formulas $\psi$  and $\phi(x, y)$
and $H^M(\beta)=H^V(\beta)$ , it follows that $H^M(\beta)\models$``$\psi$'',  and consequently $F$ and $j(F)$ are identical on the interval $[\omega, \beta)$;
in particular $\beta$ is closed under $j(F)$. Further we have $\MPB^M_\beta= \MPB_\beta$, and so $G_\beta= G \cap \MPB_\beta$ is generic over
$\MPB^M_\beta.$ As $\MPB_\beta$ is $\beta^{++}$-c.c,  we have $H^{M[G_\beta]}(\beta^{++})=H^{V[G_\beta]}(\beta^{++})$.

Let $i$ and $i^M$ enumerate closure points of $F$ and $j(F)$ respectively and suppose that $\bar{\beta} \leq \beta$ is such that $\beta=i_{\bar \beta}=i^M_{\bar \beta}$.
The singularity of $\beta$ in $M$ implies that the next step of
the iteration, the product $\MPB^M(\beta)$ in $M[G_\beta]$, is trivial at $\beta$, and so $\MPB^M(\beta)$ is the Easton-supported product of Cohen forcings in the interval
$[\beta^+, i^M_{\bar \beta+1})$, where $i^M_{\bar \beta+1}< j(\kappa)$ is the
next closure point of $j(F)$ after $\beta$. Write $\MPB^M(\beta) = \MPB^M(\beta)_1 \times \MPB^M(\beta)_2$, where
\[
\MPB^M(\beta)_1 = (\Add(\beta^+, j(F)(\beta^+)) \times  \Add(\beta^{++}, j(F)(\beta^{++})))_{M[G_\beta]}
\]
and
\[
\MPB^M(\beta)_2=(\prod_{\beta^{++} < \lambda < i^M_{\bar \beta+1}, \lambda \text{~regular}} \Add(\lambda, j(F)(\lambda))~~)_{M[G_\beta]}.
\]
By the remarks before Lemma 3.9 from \cite{friedman-honzik}, we can find $g^+_1 \in V[G]$
which is $\MPB^+(\beta)_1$-generic over $V[G_\beta]$, for some forcing notion $\MPB^+(\beta_1)_1$, such that
$g^+_1 \cap M[G_\beta]$ is $\MPB^M(\beta)_1$-generic over $M[G_\beta].$
Further we have $V[G_\beta] \cap ^{\kappa}M[G_\beta] \subseteq M[G_\beta]$ (see \cite{friedman-honzik}, Lemma 3.14), and by  arguments similar to Lemma 3.9
from \cite{friedman-honzik}, there exists $g^+_2 \in V[G]$
which is $\MPB^+(\beta)_2$-generic over $M[G_\beta]$. By Easton's theorem $g^+_1 \times g^+_2$ is in fact $\MPB^M(\beta)$-generic over $M[G_\beta]$.

Similarly, there exists a
generic for the iteration $\MPB^M$ up to the closure point $j(\kappa)$ (see \cite{friedman-honzik} Lemma 3.15 ).
Now we argue as in \cite{friedman-honzik}.
We first lifting to
the Sacks forcing at $\kappa$  and then to the rest of the forcing
above $\kappa$. This gives us
$j^*: V[G] \rightarrow M[H],$
which is defined in $V[G].$ Note that $M[H]$
captures all subsets of $\beta$ in $V[G]$ and hence $M[H] \supseteq H^{V[G]}(\b^+)$,  as required.

But $j^*(f)(\k) < j^*(h^*)(\k)=j(h^*)(\k)=\b,$ and hence $M[H] \supseteq H^{V[G]}(j^*(f)(\k)).$ Thus $j^*$ witnesses the Shelahness of $\k$ in $V[G]$ with respect to $f$. As $f$ was arbitrary, $\k$ remains a Shelah cardinal in $V[G].$
\end{proof}

\section{An indestructibility result for Shelah cardinals}
In this section we present an indestructibility result for Shelah cardinals, which involves some kind of Prikry type forcing notions.
\begin{definition}
Let $(\PP, \leq, \leq^*)$ be a set with two partial orders so that $\leq^* \subseteq \leq.$
\begin{enumerate}
\item $(\PP, \leq, \leq^*)$ is weakly $\k$-closed, if $(\PP, \leq^*)$ is $\k$-closed.
\item $(\PP, \leq, \leq^*)$ satisfies the Prikry property, if for any $p\in \PP$ and statement $\sigma$ of the forcing language $(\PP, \leq),$ there exists $q\leq^* p$ deciding $\sigma.$
\end{enumerate}
\end{definition}
Note that any $\k$-closed forcing notion can be turned into a weakly $\k$-closed Prikry type forcing notion, simply by setting $\leq^*=\leq.$
In \cite{gitik-shelah}, it is shown that, it is consistent that a strong cardinal $\k$ is indestructible under weakly $\k^+$-closed Prikry type forcing notions.

Given a set of functions $A \subseteq$$^{\k}\k,$ call $\k$ is $A$-Shelah, if for each function $f\in A,$ there is an elementary embedding witnessing the Shelahness of $\k$ with respect to $f$; so $\k$ is a Shelah cardinal iff it is a $^{\k}\k$-Shelah cardinal. At the end of their paper, Gitik and Shelah have claimed that the same argument can be applied to show that if $A=(^{\k}\k)^{V^{\MPB}}\cap V,$ then the fact that a cardinal $\k$ is $A$-Shelah can become indestructible under any weakly $\k^+$-closed Prikry type forcing notions. The next lemma, shows that this is not in fact true.
\begin{lemma}
Suppose $\k$ is a Shelah cardinal. Then there is a weakly $\k^+$-closed Prikry type forcing notion $\PP,$ such that in $V^{\PP}$, there are no measurable cardinals  above $\k;$ in particular $\k$ does not remain $(^{\k}\k)^{V^{\MPB}}\cap V$-Shelah in the generic extension.
\end{lemma}
\begin{proof}
By Lemma 2.5 we can assume that there are no Mahlo cardinals above $wt(\k)$. Then forcing with $\Add(\k^+, wt(\k))$ is as required. In fact we can even do not change the power function and preserve all cardinals if we force with the Prikry product of forcing notions $\PP_{U_\l},$ where $\l$ ranges over all measurable cardinals in $(\k, wt(\k)),$  $U_\l$ is some fixed norml measure on $\l,$ and $\PP_{U_\l}$ is the corresponding Prikry forcing notion.
\end{proof}
Note that if $\GCH$ holds in $V$, then as $cf(wt(\kappa)) > \kappa,$ we have $wt(\kappa)^\kappa=wt(\kappa),$ and hence the forcing notion $\Add(\k^+, wt(\k))$
has size $wt(\kappa)$.
We now show that it is consistent that the full Shelahness
of $\k$ is indestructible under any weakly $\k^+$-closed Prikry type forcing notion of size $< wt(\k)$ (which is optimal by Lemma 5.2).
\begin{theorem}
Suppose $GCH$ holds and $\k$ is a Shelah cardinal. Then there exists a generic extension $V[G]$ of $V$, in which $\k$ remains a Shelah cardinal, and its Shelahness is indestructible under any weakly $\k^+$-closed Prikry type forcing notion of size $< wt(\k).$
\end{theorem}
The proof uses the existence of some  kind of Laver functions for Shelah cardinals. The following can be proved as for supercompact or Strong cardinals
\begin{lemma}
Suppose $\kappa$ is a Shelah cardinal. Then there exists a partial function $L:\k \rightarrow V_\k,$ such that for every $x\in V_{wt(\k)}$ and every $\l<wt(\k)$ with $|tc(x)| <\l,$ there exists $f: \k \rightarrow \k$ and an elementary embedding $j: V \rightarrow M \supseteq H(j(f)(\k))$ with $crit(j)=\k,$ such that $j(f)(\k)>\l$ and $j(L)(\k)=x.$ We may further suppose that $\dom(L)$ just contains inaccessible cardinals, and that for all $\l\in \dom(L), L \upharpoonright \l \subseteq V_\l.$
\end{lemma}
We are now ready to complete the proof of Theorem 5.3.
\\
{\bf Proof of Theorem 5.3.} Let
\begin{center}
$\PP_\k=\langle  \langle \PP_\a: \a\leq \k  \rangle,  \langle \lusim{\QQ}_\a: \a<\k \rangle  \rangle$
\end{center}
be an Easton support Prikry iteration of length $\k,$ such that
\begin{enumerate}
\item If $\a \notin \dom(L),$ then $\Vdash_{\PP_\a}$``$\lusim{\QQ}_\a$ is the trivial forcing'',
\item Suppose $\a \in \dom(L)$ and $L(\a)=(\lusim{\QQ}, \l),$ where $\lusim{\QQ}\in V_\l$ is a $\PP_\a$-name. Then
\begin{itemize}
\item If $\Vdash_{\PP_\a}$``$\lusim{\QQ}$ is a weakly $\a^+$-closed Prikry type forcing notion, then $\lusim{\QQ}_\a=\lusim{\QQ}$'',
\item Otherwise $\Vdash_{\PP_\a}$``$\lusim{\QQ}_\a$ is the trivial forcing''.
\end{itemize}
\end{enumerate}
Let $G$ be $\PP_\k$-generic over $V$. We show that the model $V[G]$ is as required. Thus fix in $V[G]$, a weakly $\k^+$-closed Prikry type forcing notion $\QQ\in V[G]_{wt(\k)},$ let $H$ be $\QQ$-generic over $V[G]$, and let  $f: \k \rightarrow \k, f\in V[G][H].$   Let   $\l$ be such that
\begin{enumerate}
\item $\k < \l < wt(\k),$
\item  $\QQ\in V[G]_{\l},$
\item $\l$ is an inaccessible cardinal of $V$.
\end{enumerate}
As $\PP_\k$ satisfies the $\k$-c.c., and $\QQ$ adds no new functions from $\k$ to $\k,$ there is some $g:\k \rightarrow \k$ in $V$ which dominates $f$, and we can further suppose that for all $\a <\k, g(\a) > f(\a)+\omega$. Now, using Lemma 5.4, let $h\in V, h: \k \rightarrow \k$ and $j: V \rightarrow M \supseteq V_{j(h)(\k)}$ with $crit(j)=\k$ be such that:
\begin{enumerate}
\item For some inaccessible cardinal $i(\a) > g(\a), h(\a)$ is the least $\gimel$-fixed point above  $i(\a)$
of uncountable cofinality,
\item $j(i)(\k) > \l,$
\item $j(L)(\k)=(\lusim{\QQ}, j(i)(\k)),$ where $\lusim{\QQ}$ is some fixed $\PP_\k$-name for $\QQ$ with $\lusim{\QQ}\in V_\l \subseteq V_{j(i)(\k)}.$
\end{enumerate}
The proof of the main theorem in \cite{gitik-shelah} shows that we can lift $j$ to some $j^*: V[G][H] \rightarrow M^* \supseteq V[G][H]_{j(h)(\k)}.$
As requested by the referee, we now provide some details.

 Set $\PP^M_{j(\kappa)} = j(\PP_\kappa)$ and
\[
\PP^M_{j(\k)}=\langle  \langle \PP^M_\a: \a\leq j(\k)  \rangle,  \langle \lusim{\QQ}^M_\a: \a<j(\k) \rangle  \rangle.
\]
We also assume that $j$ is generated by an extender $E= \langle E_a: a \in [j(h)(\kappa)]^{<\omega}    \rangle$.
In $M$, $\PP_\kappa$ forces ``$\lusim{\QQ}$ has the Prikry property'', and so  we can factor $\PP^M_{j(\kappa)}$ as
$\PP^M_{j(\kappa)}= \PP_\kappa * \lusim{\QQ} * \lusim{\PP}^M_{\text{tail}}$ (and hence $\QQ^M_\kappa=\QQ$). The next two claims can be proved as in
\cite{gitik-shelah}.
\begin{claim}
In $V$, there is a sequence $\langle  \lusim{D}_\alpha: \alpha < \kappa^+      \rangle$
such that
\begin{enumerate}
\item [(a)] In $M$, $\Vdash_{\kappa+1}$``$\lusim{D}_\alpha$ is a $\leq^*$-dense open subset of $\lusim{\PP}^M_{\text{tail}}$''.

\item [(b)] If $\lusim{D} \in M$ and $\Vdash^M_{\kappa+1}$``$\lusim{D}$ is a $\leq^*$-dense open subset of $\lusim{\PP}^M_{\text{tail}}$'',
then for some $\alpha < \kappa^+, \Vdash^M_{\kappa+1}$``$\lusim{D}_\alpha \subseteq \lusim{D}$''.
\end{enumerate}
\end{claim}
\begin{proof}
For any name $\lusim{D}$ as in $($b$)$, if $l$ is such that $\lusim{D}=j(l)(a),$ for some $a \in [j(h)(\kappa)]^{<\omega}$,
then set
\[
\tau_{\lusim{D}} = \bigcap_{b \in [j(h)(\kappa)]^{<\omega}}  j(l)(b).
\]
As $\PP^M_{\text{tail}}$ is $\leq j(h)(\kappa)$-weakly closed, so
\begin{center}
$\Vdash_{\kappa+1}$``$\tau_{\lusim{D}}$ is a $\leq^*$-dense open subset of $\lusim{\PP}^M_{\text{tail}}$ and $\tau_{\lusim{D}} \subseteq \lusim{D}$''.
\end{center}
 It is also easily seen that
$\langle  \tau_{\lusim{D}}: \lusim{D}$ as in    $($b$)    \rangle $
has size $\kappa^+$ which completes the proof of the claim (see \cite{gitik-shelah}, Lemma 2.2).
\end{proof}
Set $X= \{ j(l)(\kappa): l: \kappa \to V          \}$. So $X \prec M$ and $^\kappa$$X \subseteq X$.
The next claim follows easily from the fact that $\PP^M_{\text{tail}}$ is $\leq j(h)(\kappa)$-weakly closed.
\begin{claim}
There exists a sequence $\langle \lusim{r}_\alpha: \alpha < \kappa^+  \rangle$ such that each
$\lusim{r}_\alpha \in X$ and $\Vdash^M_{\kappa+1}$``$\lusim{r}_\alpha \in \lusim{D}_\alpha$ and the sequence $\langle \lusim{r}_\alpha: \alpha < \kappa^+  \rangle$
is $\leq^*$-decreasing''.
\end{claim}
In $V[G*H],$ let $G_{\text{tail}}$ be the filter on $\PP^M_{\text{tail}}$, generated by  $\langle r_\alpha: \alpha < \kappa^+  \rangle$ where
$r_\alpha=\lusim{r}_\alpha[G*H]$.
\begin{claim}
If $D \subseteq \PP^M_{j(\kappa)}$ is $\leq^*$-dense, then $D$ meets $G*H* G_{\text{tail}}.$
\end{claim}
\begin{proof}
Let $\lusim{\bar D}$ be such that
\begin{center}
$(p, \lusim{q}) \in G*H$ and $(p, \lusim{q}, \lusim{r}) \in D \Longrightarrow ~~ \Vdash_{\kappa+1}$``$\lusim{r} \in \lusim{\bar D}$''
\end{center}
Then $\Vdash_{\kappa+1}$``$\lusim{\bar D}$ is $\leq^*$-dense'', and thus $r_\alpha \in \lusim{\bar D}[G*H],$
for some $\alpha < \kappa^+.$ This means that there are  $(p, \lusim{q}) \in G*H$ and $\lusim{r}$
such that $\Vdash_{\kappa+1}$``$\lusim{r}=\lusim{r}_\alpha$'' and $(p, \lusim{q}, \lusim{r}) \in D$.
The result follows.
\end{proof}
Suppose $a \in [j(h)(\kappa)]^{<\omega_1}$ and $\alpha=otp(a)$. Note that $a \in V[G*H]_{j(h)(\kappa)}$ (as $cf(j(h)(\kappa))> \omega$),
so let
$\lusim{a} \in V_{j(h)(\kappa)}$ be a name for $a$ such that
\[
\Vdash^M_{\PP^M_{j(\kappa)}} \text{~``} \lusim{a} \subseteq \lambda \text{~and~} \alpha=otp(\lusim{a})  \text{''}.
\]
Working in $V[G*H]$, define $E^*_a$ on $[\kappa]^{\alpha}$ as follows: for $A \subseteq [\kappa]^{\alpha},$
\[
A \in E^*_a \iff \exists~ \PP_\kappa\text{-name~}\lusim{A} \text{~for~}A \text{~such that there is ~}(p, \lusim{q}, \lusim{r}) \Vdash \text{``}\lusim{a} \in j(\lusim{A})\text{''~ in~} G*H*G_{\text{tail}}.
\]
Each $E^*_a$ is a $\kappa$-complete ultrafilter on $[\kappa]^{\alpha}$, further if $a \in V$ is finite, $E^*_a \supseteq E_a$.
Let $j^*_a: V[G*H] \to M^*_a \simeq \Ult(V[G*H], E^*_a)$ be the corresponding ultrapower embedding and for
$a \subseteq b$ let $k^*_{a,b}: M^*_a \to M^*_b$ be such that $j^*_b=k^*_{a,b} \circ j^*_a$. Let
\[
\langle M^*, \langle k^*_{a, \infty}: a \in [j(h)(\kappa)]^{<\omega_1}  \rangle   \rangle  =\limdir \langle \langle M^*_a: a \in [j(h)(\kappa)]^{<\omega_1} \rangle,
\langle k^*_{a,b}: a \subseteq b          \rangle  \rangle,
\]
where $k^*_{a, \infty}: M^*_a \to M^*.$ Also let $j^*: V[G*H] \to M^*$ be the direct
limit embedding. $M^*$ is easily seen to be well-founded (use the fact that $[j(h)(\kappa)]^{<\omega_1}$ is closed under countable unions)
and if we restrict only to $E^*_a$
a for finite $a$, the smaller direct limit embeds into the
full direct limit and is therefore well-founded.  From now on, let $M^*$ denote the smaller
direct limit; accordingly each $E^*_a$
is now given by the usual extender definition and $j^*$
is the ultrapower embedding. Now using Lemmas 2.4 and 2.5 from \cite{gitik-shelah}, $j^*$ is seen to be a $j(h)(\kappa)$-strongness embedding, as required.

Let's go back to the proof of Theorem 5.3.
Since $j(f)(\k) < j(g)(\k) < j(i)(\k),$ so $j^*$ witnesses the Shelahness of $\k$ in $V[G][H]$ with respect to $f$. Since $f$ was arbitrary, we conclude that $\k$
remains a Shelah cardinal in $V[G][H],$ and the theorem follows. \hfill$\Box$

We close the paper with the following question.
\begin{question}
Let $\k$ be a Shelah cardinal. Is there a generic extension in which $\k$ is the unique Shelah cardinal.
\end{question}
We may note that the answer to the question is trivial if the Shelah cardinals below $\k$ are bounded in $\k.$ The problem becomes difficult when $\{wt(\l): \l<\k$ is a Shelah cardinal$\}$ is unbounded in $\k$ (see Theorem 2.8).

\subsection*{Acknowledgements}
The author thanks the referee of the paper for many helpful comments and corrections.

School of Mathematics, Institute for Research in Fundamental Sciences (IPM), P.O. Box:
19395-5746, Tehran-Iran.

E-mail address: golshani.m@gmail.com

\end{document}